\documentclass{amsart}

\usepackage{amsmath, amsthm, amscd, amsfonts, amssymb, enumerate, verbatim, newlfont, calc, graphicx, color}
\usepackage[bookmarksnumbered, colorlinks, plainpages]{hyperref}
\usepackage{tikz}
 \usepackage{doi}

\newtheorem{theorem}{Theorem}[section]
\newtheorem{question}[theorem]{Question}
\newtheorem{lemma}[theorem]{Lemma}
\newtheorem{proposition}[theorem]{Proposition}
\newtheorem{corollary}[theorem]{Corollary}
\theoremstyle{definition}
\newtheorem{definition}[theorem]{Definition}

\newtheorem{example}[theorem]{Example}

\usepackage{pstricks}
\usepackage{epsfig}
\usepackage{pst-grad}
\usepackage{pst-plot}

\begin{document}
\author[M. Nasernejad, S. Bandari,   and  L. G. Roberts]{Mehrdad ~ Nasernejad$^{1,*}$,   Somayeh Bandari$^{2}$,   and Leslie G. Roberts$^{3}$}
\title[Closed  neighborhood ideals and dominating ideals]{Normality and associated primes of Closed  neighborhood ideals and dominating ideals}
\thanks{$^*$Corresponding author}

\thanks{E-mails:  m$\_$nasernejad@yahoo.com, s.bandari@bzte.ac.ir, robertsl@queensu.ca}
\maketitle

\begin{center}
{\it
$^{1}$Univ. Artois, UR 2462, Laboratoire de Math\'{e}matique de  Lens (LML), \\  F-62300 Lens, France \\
%\item \href{https://orcid.org/0000-0003-1073-1934}{\textcolor{orcidlogocol}{\aiOrcid} \hspace{2mm}orcid.org/0000-0003-1073-1934}\\
  $^2$Department of Mathematics, Buein Zahra Technical University, \\
  Buein Zahra, Qazvin, Iran\\
  $^3$Department of Mathematics and Statistics,
Queen's University, \\
Kingston, Ontario, Canada, K7L 3N6
}
\end{center}

\vspace{0.4cm}

\begin{abstract}
In this paper, we first give some  sufficient  criteria for normality of monomial ideals. As applications, we show that closed
neighborhood ideals of complete bipartite graphs are normal, and hence  satisfy the (strong) persistence property.
We also prove that  dominating ideals of complete bipartite graphs are nearly normally torsion-free. In addition,  we show that dominating ideals of $h$-wheel graphs, under certain condition,  are normal.
\end{abstract}
\vspace{0.4cm}

%+++++++++++++++++++++++++++++++++++++++++++++
%+++++++++++++++++++++++++++++++++++++++++++++
\section{Introduction}

Broadly speaking, monomial ideals play an essential role in studying the connection between combinatorics and  commutative algebra. In fact, the relation between these two  fields  enables  us to employ  techniques and methods in commutative algebra to probe   combinatorial problems, and vice versa. Hence,  commutative algebraists have initiated  exploring the properties of   combinatorial objects such as  graphs, hypergraphs, simplicial complexes, and posets 
 through monomial ideals. 

One of the innovators  in this area is  Villarreal \cite{V} who  introduced the concept  of edge ideals. Indeed, let $G=(V(G), E(G))$  be a finite simple graph on the vertex set $V(G)=\{1,\ldots, n\}$, that is, $G$ has no loops and no multiple edges. In addition,  let  $R = K[x_1,\ldots,x_n]$ be a polynomial ring over a field $K$  in $n$ variables.
Then  the edge ideal $I(G) \subset K[x_1,\ldots,x_n]$  is generated by all monomials $x_ix_j$ such that $\{i,j\}\in E(G)$.
Furthermore, the cover  ideal of $G$, denoted by $J(G)$,  is generated by monomials that correspond to vertex covers of $G$, where  a vertex cover is a  set of  vertices that contains  at least one vertex from each edge. It is well-known that $J(G)$ is the Alexander dual of $I(G)$, that is, $J(G)=I(G)^\vee$.
The  algebraic invariants of edge and cover ideals have been investigated by many researchers. For example, it  has already  been proved  in \cite{SVV} that the edge ideals of bipartite graphs are normal and also  the cover ideals of perfect graphs are normal   \cite[Corollary 14.6.25]{V1}. Recall  that a graph is perfect  if  and only if the graph and its complement do not contain any odd cycle of length at least five as an induced subgraph \cite[Theorem 14.18]{BM}.
On the other hand, generally, identifying  classes of monomial ideals which have the (strong) persistence property is hard. Based on  \cite[Example  7.7.18]{V1},  there exist  square-free monomial ideals which do  not satisfy the persistence property.  It follows  from  \cite[Theorem 7.7.14]{V1} that all edge ideals of finite simple graphs have the strong persistence property; moreover, this result is true for every  finite graph with loops \cite{RT}.   In addition, according to   \cite{FHV2}, the cover ideals of perfect graphs   satisfy the persistence property.  Furthermore, more recently, it has been shown that the cover ideals of theta graphs and Jahangir's graphs are normal and satisfy the  strong persistence property, consult \cite{ANR1, ANR2}.

In 2020,   Sharifan and Moradi  \cite{SM} continued  to  explore   the connection between combinatorics and commutative algebra by  introducing  the notion of closed neighborhood ideals and dominating ideals of graphs. In particular, they   calculated some algebraic invariants of  these  ideals such as regularity and projective dimension in terms of the information from the underlying graph. Next, Honeycutt  and    Sather-Wagstaff, in \cite{HS},  investigated the   Cohen-Macaulay, unmixed, and complete intersection properties of closed neighborhood ideals.

 In this paper, we concentrate on the normality,  strong persistence property, persistence property, and symbolic strong persistence property of closed neighborhood ideals and dominating ideals of some classes of  graphs. Although, it should be observed that most of the results in this paper
are about the normality of the ideals, which then implies the persistence and strong persistence properties as consequences.
 In particular, one can detect that some properties of edge ideals or cover ideals cannot be developed  to closed neighborhood ideals  or  dominating ideals of graphs. For example, according to \cite[Lemma 4.2]{NQ}, the dominating ideal of the cycle graph $C_9$ is  normally torsion-free, while both of the edge ideal and cover ideal of $C_9$ are not normally torsion-free. This shows that studying of algebraic invariants   of    closed neighborhood ideals and dominating ideals of graphs are important and we cannot extend the results related to edge or cover ideals to closed neighborhood ideals or  dominating ideals. 
  One of our main aims in  this paper is to  introduce some  classes of graphs such that  their closed neighborhood ideals or dominating ideals are normal and satisfy the (symbolic) (strong) persistence property.

This paper is organized as follows.  In Section \ref{Preliminaries}, we give all required definitions, notions, and properties which we need throughout this paper.
In section  \ref{Criteria}, we supply several  sufficient   criteria which help us to detect the normality of monomial ideals. The results of this section will be used in studying  closed  neighborhood ideals and dominating ideals of  some classes of graphs. The main results of this section  are  as follows:

 \noindent{\bf  Theorem 3.1.}
Let $I$ and $H$  be two  normal  monomial ideals in a  polynomial ring $R=K[x_1, \ldots, x_n]$
such that $I+H$ is normal.
  Let  $x_{d} \in \{x_1, \ldots, x_n\}$ be  a variable with
 $\mathrm{gcd}(v,x_{d})=1$  for all $v\in \mathcal{G}(I)\cup  \mathcal{G}(H)$, and $c$ be a positive  integer number.
 Then the following statements hold:
\begin{itemize}
\item[(i)]   $L:=I+x^c_{d}H$  is normal.
\item[(ii)]  $L:=I+x^c_{d}H$  has the strong persistence property.
\item[(iii)]  $L:=I+x^c_{d}H$  has the persistence property.
\item[(iv)]  $L:=I+x^c_{d}H$  has the symbolic persistence property.
\end{itemize}
 
  \noindent{\bf  Proposition  3.10.}
Let $I \subset  R=K[x_{1},\ldots ,x_{n}, x_{n+1}]$ be a normal square-free monomial ideal
such that $I \cap (x_n)  + (I:_Rx_{n+1})$ is normal. Then   $L:=I \cap (x_{n},x_{n+1}) $ is normal.

\bigskip
Section \ref{c-b graphs} is concerned with closed  neighborhood ideals and dominating ideals of complete bipartite graphs. We show that closed  neighborhood ideals of complete bipartite graphs are normal and have the (strong) persistence property
(Theorem \ref{NI-Bipartite}). In addition,   we  explore when the irrelevant  ideal appears in the set of associated primes  of powers of closed neighborhood ideals of complete bipartite graphs (Proposition \ref{Maximal-Bipartite}).
One of the main result of this section is as follows:

 \noindent{\bf  Theorem  4.4.}
The dominating  ideals of  complete bipartite graphs are  nearly normally torsion-free.
 
 \bigskip
 Section \ref{wheel} is devoted to studying dominating ideals of $h$-wheel graphs. We prove that dominating ideals of
cycle graphs  are normal, see Corollary \ref{Normality-Cycle}. In addition, the following theorem, which is one of the main result of this section, 
 states that  dominating ideals of $h$-wheel graphs  that, under certain condition,  are normal:

 \noindent{\bf  Theorem  5.8.}
Let $G$ be an $h$-wheel graph  with  rim  $R^G$ and   center  $C^G$ such that
 $V(C^G)=\{y_1, \ldots, y_h\}$ and $V(R^G)=\{x_1, \ldots, x_{2m+1}\}$, where $m\geq 2$.  Also, let  $x_{\ell_1}, \ldots, x_{\ell_k}$ be the radial vertices such that  there exist at least three consecutive numbers among them. Let $DI(G)$ denote the  dominating ideal of $G$. Then the following statements hold:
\begin{itemize}
\item[(i)]   $DI(G)$  is normal.
\item[(ii)]  $DI(G)$  has the strong persistence property.
\item[(iii)]  $DI(G)$  has the persistence property.
\end{itemize}

\bigskip
Throughout this paper,  we denote the unique minimal set of monomial generators of a  monomial ideal $I$  by $\mathcal{G}(I)$. Also, $R=K[x_1,\ldots, x_n]$ is a polynomial ring over a field $K$ and $x_1, \ldots, x_n$ are indeterminates.
A simple graph $G$ means that $G$ has  no loop and no multiple edge. All graphs in this paper are undirected.
Moreover, if $G$ is a finite simple graph, then $NI(G)$  (respectively, $DI(G)$)  stands for the closed neighborhood ideal (respectively, dominating ideal) of  $G$.

%++++++++++++++++++++++++++++++++++++++++++++++++++++++++++++++++++++++++++++++++++++++

\section{Preliminaries} \label{Preliminaries}

In this section we review the definitions and notions which we need in the rest of this paper. We begin with the definition of nomality.

\begin{definition}
Let $S$ be a ring and $I$ be an ideal in $S$. An element $f\in S$ is {\it integral} over $I$, if there exists an equation
 $$f^k+c_1f^{k-1}+\cdots +c_{k-1}f+c_k=0 ~~\mathrm{with} ~~ c_i\in I^i.$$
 The set of elements $\overline{I}$ in $S$ which are integral over $I$ is the
 {\it integral closure} of $I$. The ideal $I$ is  called {\it integrally closed}, if $I=\overline{I}$, and $I$ is said to be  {\it normal} if all powers of $I$ are integrally closed.
 \end{definition}
 If   $I$ is  a monomial ideal in a
polynomial ring $R$, then  $\overline{I}$ is the monomial ideal generated by all monomials $u \in R$ for which there exists an integer $k$ such that   $u^{k}\in I^{k}$,
refer to    \cite[Theorem 1.4.2]{HH1}.  \par
  Here, we turn our attention to  notions which are related to the  associated primes of powers of ideals.
  Suppose that  $R$ is  a commutative Noetherian ring and $I$  an ideal of $R$. A prime ideal $\mathfrak{p}\subset  R$ is an {\it associated prime} of $I$ if there exists an element $v$ in $R$ such that $\mathfrak{p}=(I:_R v)$, where $(I:_Rv)=\{r\in R~:~ rv\in I\}$. The   set of associated primes  of $I$, denoted by  $\mathrm{Ass}_R(R/I)$, is the set of all prime ideals associated to  $I$. 
    In \cite{BR},  Brodmann  proved  that the sequence $\{\mathrm{Ass}_R(R/I^k)\}_{k \geq 1}$ of associated prime ideals is stationary for large $k$, that is to say,  there exists a positive integer $k_0$ such that  $\mathrm{Ass}_R(R/I^k)=\mathrm{Ass}_R(R/I^{k_0})$ for all integers $k\geq k_0$. The minimal   such $k_0$ is called the {\it index of stability}  of $I$ and $\mathrm{Ass}_R(R/I^{k_0})$ is called the {\it stable set}  of associated prime ideals of $I$, which is denoted by $\mathrm{Ass}^{\infty }(I)$.
  \begin{definition} 
An ideal $I$ of $R$ satisfies the {\it persistence property} if
 $\mathrm{Ass}_R(R/I^k) \subseteq \mathrm{Ass}_R(R/I^{k+1})$ for all positive integers $k$.
 \end{definition}
 \begin{definition}(\cite{N2})
An ideal $I$ of $R$ has the {\it strong persistence property} if $(I^{k+1}:_R I)=I^k$ for all  $k\geq 1$.
\end{definition}
\begin{definition}(\cite{KNT, NKRT})
We say that $I$  has the {\it symbolic strong persistence property}  if
$(I^{(k+1)}: I^{(1)})=I^{(k)}$ for all $k$, where $I^{(k)}$ denotes the  $k$-th symbolic  power  of $I$.
\end{definition}
\begin{definition}(\cite[Definition 2.1]{Claudia})
 A monomial ideal $I$ in a polynomial  ring $R$  is called {\it nearly normally torsion-free}  if there exist a positive integer $k$ and a monomial prime ideal  $\mathfrak{p}$ such that $\mathrm{Ass}_R(R/I^m)=\mathrm{Min}(I)$ for all $1\leq m\leq k$, and
 $\mathrm{Ass}_R(R/I^m) \subseteq \mathrm{Min}(I) \cup \{\mathfrak{p}\}$ for all $m \geq k+1$, where  $\mathrm{Min}(I)$  denotes  the set of minimal prime ideals of $I$.
 \end{definition}
 It should be noted that, according to   \cite[Theorem 6.2]{RNA}, every normal monomial ideal has the strong persistence property. Also, it follows from
  \cite[Proposition 2.1]{NQBM} that  the strong persistence property implies the persistence property. Hence, we deduce that  normal $\Rightarrow$ strong persistence property $\Rightarrow$ persistence property. In addition, one can derive  from \cite[Theorem 11]{RT} that  the strong persistence property yields  the symbolic strong persistence property.

Suppose that  $G$ is  a finite simple graph with the vertex set $V(G)=\{1, \ldots, n\}$ and  the edge set $E(G)$. The {\it closed neighborhood} of a vertex $v\in V(G)$ is  $N_G[v]=\{v\} \cup \{u~:~\{u,v\}\in E(G)\}.$  The following definitions and results come from \cite{SM}.  
\begin{definition}
The  {\it  closed neighborhood ideal} of $G$, denoted by $NI(G)$,  has been defined as
$$NI(G)= (\prod_{j \in N_G[i]} x_j:  i \in V(G))\subset R=K[x_1, \ldots, x_n].$$
A subset $S\subseteq V(G)$ is called a {\it dominating set}
 of $G$ if $S\cap N_G[v]\neq \emptyset$ for any $v\in V(G)$.
 Furthermore, $S$ is called a  {\it minimal dominating set} of $G$  if it is a dominating set of $G$  and  no proper subset of $S$ is a dominating set of $G$.  The {\it dominating ideal} of $G$ has been  defined as
 $$DI(G)=(\prod_{i \in S}x_i: S \text{ is a minimal dominating set of } G) \subset R=K[x_1, \ldots, x_n].$$
\end{definition}
 If  $u=x_{1}^{a_{1}}\cdots x_{n}^{a_{n}}$ is  a
monomial in a polynomial ring $R$, then
the {\it support} of $u$ is given by $\mathrm{supp}(u):=\{x_{i}|~a_{i}>0\}$. In addition, for a monomial ideal $I$, we set $\mathrm{supp}(I):=\bigcup_{u \in \mathcal{G}(I)}\mathrm{supp}(u)$.
We recall  that    for any  square-free monomial ideal $I\subset R$,  the {\it Alexander dual} of  $I$, denoted by $I^\vee$,  is given by
 $$I^\vee= \bigcap_{u\in \mathcal{G}(I)} (x_i~:~ x_i \in \mathrm{supp}(u)).$$

 By virtue of \cite[Lemma 2.2]{SM}, we have $DI(G)$ is the Alexander dual of $NI(G)$, that is, $DI(G)=NI(G)^\vee$.

%%%%%%%%%%%%%%%%%%%%%%%%%%%%%%%%%%%%%%%%%%%%%%%%%%%%%%%%
\section{Some   sufficient  criteria for normality of monomial ideals}
\label{Criteria}
This section is devoted to  present some  sufficient  criteria  which help us to detect the normality of monomial ideals. The results of this section will be used in studying  closed  neighborhood ideals and dominating ideals of  some classes of  graphs.

%%%%%%%%%%%%%%%%%%%%%%%%%%%%%%%%%%%%%%%%%%%%%%%%%%
 The following theorem is essential for us to verify Theorem \ref{I+hH}  and   Proposition \ref{Pro.vI+wJ}.
 It should be noted that Theorem \ref{I+xcH} is an updated version of \cite[Theorem 3.1]{NQBM}. In fact, in Theorem \ref{I+xcH}, we focus on the monomial ideals, while in \cite[Theorem 3.1]{NQBM}, the authors argued on square-free monomial ideals.

\begin{theorem} \label{I+xcH}
Let $I$ and $H$  be two  normal  monomial ideals in a  polynomial ring $R=K[x_1, \ldots, x_n]$
such that $I+H$ is normal.
  Let  $x_{d} \in \{x_1, \ldots, x_n\}$ be  a variable with
 $\mathrm{gcd}(v,x_{d})=1$  for all $v\in \mathcal{G}(I)\cup  \mathcal{G}(H)$, and $c$ be a positive  integer number.
 Then the following statements hold:
\begin{itemize}
\item[(i)]   $L:=I+x^c_{d}H$  is normal.
\item[(ii)]  $L:=I+x^c_{d}H$  has the strong persistence property.
\item[(iii)]  $L:=I+x^c_{d}H$  has the persistence property.
\item[(iv)]  $L:=I+x^c_{d}H$  has the symbolic persistence property.
\end{itemize}
 \end{theorem}

\begin{proof}
(i) Assume that  $\mathcal{G}(I)= \{u_{1},\ldots ,u_{s}\}$ and
$\mathcal{G}(H)=\{h_1, \ldots, h_r\}$.
In the light of   $\mathrm{gcd}(v,x_{d})=1$ for all $v\in \mathcal{G}(I)\cup  \mathcal{G}(H)$, without loss of generality, one can
assume that $x_{d}=x_{1}\in K[x_{1}]$ and
$$\mathcal{G}(I) \cup \mathcal{G}(H)=\{u_{1},\ldots ,u_{s}, h_1, \ldots, h_r\}\subseteq K[x_{2},\ldots ,x_{n}].$$
To establish the normality of $L$, one has to prove that
 $\overline{L^{t}}=L^{t}$ for all
integers $t\geq 1$.  Obviously, $L^t \subseteq \overline{L^{t}}$, so it  is enough to show
 that $%
\overline{L^{t}}\subseteq L^{t}$. For this  it  suffices to show that an  arbitrary monomial  $\alpha \in  \overline{L^{t}}$ is in $L^t$.
Write $\alpha =x_1^{b}\delta $ with $x_1\nmid \delta $ and $\delta \in R$. In view of
\cite[Theorem 1.4.2]{HH1}, one has   $\alpha ^{k}\in L^{tk}$ for some integer $k\geq 1$. Hence, one can conclude that
$\alpha^{k} \in I^{p}(x^c_{1}H)^{q}$ for some $p$ and $q$ with $p+q=tk$. If $p =0$ (respectively, $q =0$), then
$x_1^{bk}\delta^{k}\in (x^c_1H)^{tk}$ (respectively, $x_1^{bk}\delta^{k}\in I^{tk}),$
and so $x_1^{b}\delta\in \overline{(x^c_1H)^{t}}$
(respectively, $x_1^{b} \delta\in\overline{I^{t}}).$
Because   $I$ and $H$ are normal, this implies that
$x_1^{b}\delta\in {(x^c_1H)^{t}}$
(respectively, $x_1^{b} \delta\in{I^{t}}),$
and the proof is complete. Accordingly, one can  assume that  $p>0$ and $q>0$.
Then,  we   can write
\begin{equation}
\alpha^{k}=x_1^{bk}\delta^{k}=
\prod\limits_{i=1}^{s}u_{i}^{p_{i}}
x_1^{cq+\varepsilon}
\prod\limits_{j=1}^{r}h_{j}^{q_{j}}\beta \text{,} \label{55}
\end{equation}%
with $\sum_{i=1}^{s}p_{i}=p$, $\sum_{j=1}^{r}q_j=q$,
$p+q=tk$, $\varepsilon \geq 0$, and $\beta$ is some
monomial in $R$ such that $x_1\nmid \beta$.
As  $x_1\nmid \beta$, $x_1\nmid \delta$,  and
$\mathrm{gcd}(v,x_1)=1$ for all $v\in \mathcal{G}(I)\cup
 \mathcal{G}(H)$, we obtain
$bk=cq+\varepsilon $, in particular, $bk \geq cq$. In addition,   suppose that $p$  is minimal according to the  membership $\delta^{k}\in I^p H^{q}$. Now,  one can deduce from  (\ref{55}) that
\begin{equation}
\delta^{k}=\prod\limits_{i=1}^{s}u_{i}^{p_{i}}
\prod\limits_{j=1}^{r}h_{j}^{q_{j}}\beta
 \in (I+H)^{tk}.\label{66}
\end{equation}%
  It follows from    (\ref{66}) that   $\delta \in \overline{(I+H)^{t}}.$ Thanks to     $I+H$ is
normal, one has   $\overline{(I+H)^{t}}=(I+H)^{t}$, and so $\delta \in (I+H)^{t}$. Thus, we get $\delta \in I^lH^z$ for some $l$ and $z$ with $l+z=t$, in particular, $\delta^{k} \in I^{lk}H^{zk}$.  One can derive from the minimality of $p$ that $p\leq lk$. On account of $\delta \in I^lH^z$, we can write%
\begin{equation}
\delta =\prod\limits_{i=1}^{s}u_{i}^{l_{i}}
\prod\limits_{j=1}^{r}h_{j}^{z_{j}}
\gamma \text{,} \label{77}
\end{equation}
with $\sum_{i=1}^{s}l_{i}=l$, $\sum_{j=1}^{r}z_{j}=z$, $l+z=t$,  and $\gamma $ is some monomial in
$R$. Observe  that $x_1\nmid \gamma $ due to   $x_1\nmid \delta $.  \par
We have shown above that  $p+q=tk$, $lk+zk=tk$, and $p \leq lk$.
 It follows from this that $q\geq zk$. We have also shown above that  $bk\geq cq$. Thus,  $bk\geq czk$, and since $k\geq 1$, we have $b\geq cz$.
Multiplying (\ref{77}) by $x_1^{b}$ yields
 $x_1^{b}\delta =\prod_{i=1}^{s}u_{i}^{l_{i}}x_1^{b} \prod_{j=1}^{r}h_{j}^{z_{j}}\gamma$. The latter can be rewritten as
 $\prod_{i=1}^{s}u_{i}^{l_{i}} \prod_{j=1}^{r} (x_1^{c}h_{j})^{z_{j}}(x_1^{b-cz}\gamma)\in I^{l}(x_1^{c}H)^{z}\subseteq L^t.$
    Therefore, $\alpha=x_1^{b}\delta\in L^t$, which makes $L$  normal, completing the proof.  \par
   (ii)-(iv) are trivial.
 \end{proof}
%%%%%%%%%%%%%%%%%%%%%%%%%%
%%%%%%%%%%%%%%%%%%%%%%%%%%%
To formulate Theorem \ref{I+hH},  one requires some auxiliary results. We first recall the following definition and result from \cite{HQ}.

\begin{definition}
Let $I\subset R=K[x_1, \ldots, x_n]$ be a monomial ideal with $\mathcal{G}(I)=\{u_1, \ldots, u_m\}$. The {\it  linear relation graph} $\Gamma_I$ of $I$ is the graph with the edge set
\[
E(\Gamma_I)=\{\{x_i,x_j\}: \text{there exist $u_k$,$u_l \in \mathcal{G}(I)$ such that $x_iu_k=x_ju_l$}\},
\]
 and the vertex set $V(\Gamma_I)=\bigcup_{\{x_i,x_j\}\in E(\Gamma)}\{i,j\}$.
 \end{definition}

\begin{theorem}\cite[Theorem 3.3]{HQ}\label{Th. Depth}
Let $I \subset R=K[x_1, \ldots, x_n]$ be a monomial ideal generated in a single degree whose linear relation graph has $r$ vertices and $s$ connected components. Then
\[
\mathrm{depth}(R/I^t) \leq n-t-1 \text{ for } t=1, \ldots, r-s.
\]
\end{theorem}

%%%%%%%%%%%%%%%%%%%%%%%%%
%%%%%%%%%%%%%%%%%%%%%%%%%
 As an  immediate consequence of Theorem \ref{Th. Depth}, we get the following corollary.

 \begin{corollary}\label{Cor. Depth}
Let $I \subset R=K[x_1, \ldots, x_n]$ be a monomial ideal generated in a single degree whose linear relation graph has $n$ vertices and one  connected component. Then
\begin{itemize}
\item[(i)] $\mathrm{depth}(R/I^{n-1})=0$. In particular, $\mathfrak{m}=(x_1, \ldots, x_n) \in \mathrm{Ass}(R/I^{n-1})$.
\item[(ii)] If $I$ satisfies the persistence property, then
$\mathrm{lim}_{k \rightarrow \infty}\mathrm{depth} R/I^k=0$. In particular, $\mathfrak{m} \in \mathrm{Ass}(R/I^{k})$ for all $k\geq n-1$.
\end{itemize}
\end{corollary}

%%%%%%%%%%%%%%%%%%%%%%%%%
%%%%%%%%%%%%%%%%%%%%%%%%%

 The following theorem, which is a generalization of   Theorem \ref{I+xcH},  is crucial for us to prove Theorem \ref{NI-Bipartite} and Lemma \ref{Normality-H}.

 \begin{theorem} \label{I+hH}
Let $I$ and $H$  be two normal   monomial ideals in a  polynomial ring $R=K[x_1, \ldots, x_n]$
such that $I+H$ is normal.   Let  $h$ be  a  monomial in $R$  with
 $\gcd (v,h)=1$ for all $v\in \mathcal{G}(I)\cup  \mathcal{G}(H)$.  Then the following statements hold:
\begin{itemize}
\item[(i)]   $L:=I+hH$  is normal.
\item[(ii)]  $L:=I+hH$  has the strong persistence property.
\item[(iii)]  $L:=I+hH$  has the persistence property.
\item[(iv)]  $L:=I+hH$  has the symbolic persistence property.
\item[(v)]  If  $L:=I+hH$ is generated in a single degree whose linear relation graph has $n$ vertices and one  connected component, then  $\mathrm{lim}_{k \rightarrow \infty}\mathrm{depth} R/L^k=0$. In particular, $\mathfrak{m}=(x_1, \ldots, x_n) \in \mathrm{Ass}(R/L^{k})$ for all $k\geq n-1$.
\end{itemize}
\end{theorem}

\begin{proof}
(i)  Let $h=x_1^{c_1} \cdots x_m^{c_m}$  and  use induction on $m$. \par
(ii)-(iv) are trivial.  \par
 (v) Corollary \ref{Cor. Depth}  together with (iii) yield the assertion.
\end{proof}

 %%%%%%%%%%%%%%%%%%%%%%%%%
%%%%%%%%%%%%%%%%%%%%%%%%%
To show Theorem \ref{DI-Wheel}, we need to use the next proposition.

\begin{proposition} \label{I+JH}
Let $I\subseteq H$ be two normal   monomial ideals in a polynomial ring  $R=K[x_1, \ldots, x_n]$. Let  $J$ be a   monomial ideal of $R$ such that $\mathrm{gcd}(u,v)=1$  for all $v, u \in \mathcal{G}(J)$ with $u\neq v$.
Also, let $\mathrm{gcd}(u,v)=1$ for all $u\in \mathcal{G%
}(I)\cup \mathcal{G}(H)$ and $v\in \mathcal{G}(J)$.
 Then the following statements hold:
\begin{itemize}
\item[(i)]   $L:=I+JH$  is normal.
\item[(ii)]  $L:=I+JH$  has the strong persistence property.
\item[(iii)]  $L:=I+JH$  has the persistence property.
\end{itemize}
\end{proposition}

\begin{proof}
(i) Let  $\mathcal{G}(J)=\{u_{1},\ldots ,u_{s}\}$.
Proceed by induction on $s$ and use  Theorem \ref{I+hH}. \par 
(ii) and (iii) are trivial.
\end{proof}

%%%%%%%%%%%%%%%%%%%%%%%%%%%%
%%%%%%%%%%%%%%%%%%%%%%%%%%%%

It has already been remained the following open question:

\begin{question} \cite[Question 2.11]{ANKRQ}
Let $I$ be a normal square-free monomial ideal in $%
R=K[x_{1},\ldots ,x_{n}]$ with  $\mathcal{G}(I) \subset R$. Then, in general, can one conclude    $L:=IS\cap (x_{n},x^{\ell}_{n+1})\subset S=R[x_{n+1}]$ with $\ell >1$, is normal?
\end{question}

As an application of Theorem \ref{I+xcH},  we provide an affirmative answer to above  open question  in the following proposition:

%%%%%%%%%%%%%%%%%%%%%%%%%%%
%%%%%%%%%%%%%%%%%%%%%%%%%%%

 \begin{proposition} \label{Pro.vI+wJ}
Let $I \subset R=K[x_{1},\ldots ,x_{n}]$ be a normal square-free monomial ideal   with $\mathcal{G}(I) \subset R$. Then
 $L:=IS\cap (x_{n},x^{\ell}_{n+1})\subset S=R[x_{n+1}]$ with $\ell \geq 1$, is normal.
\end{proposition}

\begin{proof}
  It is routine to check that $$L=I\cap (x_n) + I\cap (x^{\ell}_{n+1}) =x_n(I:_Sx_n) + x^{\ell}_{n+1}I.$$
Since $I$ is normal, \cite[Proposition 12.2.3]{V1}  implies that $(I:_Sx_n)$ is normal, and hence $x_n(I:_Sx_n)$ is normal as well.  In addition,
$x_n(I:_Sx_n) + I=I$ is normal. Now, one can rapidly deduce from Theorem \ref{I+xcH} that $L$ is normal, as claimed.
\end{proof}

%%%%%%%%%%%%%%%%%%%%%%%%%%%
%%%%%%%%%%%%%%%%%%%%%%%%%%%

 To see an application of Proposition  \ref{Pro.vI+wJ}, we give the subsequent corollary.

\begin{corollary} \label{Cor.vI+wJ}
Let $I$ be a normal square-free monomial ideal in $%
R=K[x_{1},\ldots ,x_{n}]$ with $\mathcal{G}(I) \subset R$. Then the square-free monomial ideal $$L:=IS\cap (x_{n},x_{n+1},\ldots, x_{n+m})\subset S=R[x_{n+1}, \ldots, x_{n+m}],$$ is normal.
\end{corollary}

\begin{proof}
We demonstrate the claim by using induction on $m$. In view of Proposition \ref{Pro.vI+wJ}, we deduce  that the claim is true  for the case in which $m=1$. Suppose that the claim holds for $m-1$.
Put $J:=I\cap (x_{n},x_{n+1},\ldots, x_{n+m-1})$. We thus have $L=J+x_{n+m}I$. The inductive hypothesis implies that
$J$  is normal. In the light of Theorem \ref{I+hH}, one can conclude that $L$ is normal. This completes the inductive step, and so  the claim has been shown by induction.
\end{proof}

%%%%%%%%%%%%%%%%%%%%%%%%%%%
%%%%%%%%%%%%%%%%%%%%%%%%%%%

The next  proposition will be used to establish Lemma \ref{Normality-H}.

\begin{proposition} \label{xn&x(n+1)}
Let $I \subset  R=K[x_{1},\ldots ,x_{n}, x_{n+1}]$ be a normal square-free monomial ideal
such that $I \cap (x_n)  + (I:_Rx_{n+1})$ is normal. Then   $L:=I \cap (x_{n},x_{n+1}) $ is normal.
\end{proposition}

\begin{proof}
It is not hard to check that   $I\cap (x_{n})=x_{n}\left( I:_{R}x_{n}\right)$ and $I\cap (x_{n+1})=x_{n+1}\left( I:_{R}x_{n+1}\right)$. Hence,
in view of \cite[Exercise  6.1.23]{V1},  we get
$L=x_{n}\left( I:_{R}x_{n}\right)+x_{n+1}\left( I:_{R}x_{n+1}\right)$. To simplify notation,
put  $F:=\left( I:_{R}x_{n}\right) $ and $G:=\left( I:_{R}x_{n+1}\right)$.
Thus, $L=x_{n}F+x_{n+1}G$.
Let  $\mathcal{G}(I)=\{u_1, \ldots, u_s\}$, and define $f_i=u_i$
 if  $x_n \nmid u_i$, and $f_i=u_i/x_n$ if $x_n \mid u_i$.  By the square-free assumption
on $I$, none of the $f_i$ is divisible by $x_n$.  Similarly, define  a set of
elements $g_i$ by considering divisibility by $x_{n+1}$, i.e., $g_i = u_i$
 if $x_{n+1} \nmid  u_i$ and $g_i = u_i/x_{n+1}$ if $x_{n+1} \mid u_i$. Again, by the square-free
assumption, none of the $g_i$ is divisible by $x_{n+1}$.
One can readily see that $F=(f_1, \ldots, f_s)$ and $G=(g_1, \ldots, g_s)$.
By   \cite[Proposition 12.2.3]{V1},  $F$ and $G$ are  normal, and so  $x_{n}F$ and
$x_{n+1}G$ are also  normal.
By  hypothesis,  $x_nF+G$ is normal. Accordingly, we obtain  $\overline{\left( x_{n}F+G\right) ^{t}}=\left(
x_{n}F+G\right) ^{t}$ for all $t\geq 1$. One has to prove  that  $\overline{%
L^{t}}=L^{t}$ for all $t\geq 1$.
To accomplish this, it is enough to show that $\overline{L^t} \subseteq L^t$ for all $t\geq 1$.  To do this, pick an arbitrary  monomial $\alpha$  in $\overline{L^{t}}$ and write
$\alpha =x_{n+1}^{b}\delta $ for some integer $b$ and some monomial $\delta
\in R$ with $x_{n+1}\nmid \delta $. Since  $\alpha \in \overline{L^{t}}$,  \cite[Theorem 1.4.2]{HH1} gives  that
 $%
\alpha ^{k}\in L^{tk}=\left( x_{n}F+x_{n+1}G\right) ^{tk}$ for some positive integer $%
k$. This implies that  $\alpha ^{k}\in \left( x_{n}F\right) ^{p}\left(x_{n+1}G\right) ^{q}$ for some integers $p$ and $q$ with $p+q=tk$. Assume that $q$ is maximal according to this membership. This means that $p$ is minimal  according to this membership. Observe  that if $p=0$,  then $\alpha^{k}\in \left( x_{n+1}G\right) ^{tk}$, and hence $\alpha \in \overline{\left(x_{n+1}G\right) ^{t}}=\left( x_{n+1}G\right) ^{t}\subset L^{t}$. Henceforth,
let  $p>0$. Thanks to  $x_{n}F$ is normal, a similar argument shows that  one  may  assume $q>0$ too. Write%
\begin{equation}
\alpha ^{k}=x_{n+1}^{bk}\delta ^{k}=\prod\limits_{i=1}^{s}f_{i}^{p_{i}}%
\text{ }x_{n}^{p}\text{ }\prod\limits_{j=1}^{s}g_{j}^{q_{j}}\text{ }%
x_{n+1}^{q}\text{ }\beta \text{,}  \label{88}
\end{equation}%
with $\sum\nolimits_{i=1}^{s}p_{i}=p$, $\sum\nolimits_{j=1}^{s}q_{j}=q$,
and $\beta $   some monomial in $S$. Let  $x_{n+1}\mid \beta $. On account of either $f_{j}x_{n}=u_{j}$ or $%
f_{j}=u_{j}$, this contradicts  the maximality of $q$.
In addition, if there exists some $f_j$ with $p_j>0$ such that $x_{n+1}\mid f_j$, then once again this leads to a contradiction to the maximality of $q$.
Therefore, we can assume in (\ref{88})\ that $x_{n+1}\nmid
\beta $, and also $x_{n+1}\nmid f_j$ with $p_j>0$. Consequently,  one  can derive  that $q=bk$, in particular, $\delta^{k} \in (x_nF)^p G^q$. Now,  we get the following equality
$$
\delta ^{k}=\prod\limits_{i=1}^{s}f_{i}^{p_{i}}\text{ }x_{n}^{p}\text{ }%
\prod\limits_{j=1}^{s}g_{j}^{q_{j}}\text{ }\beta \in \left( x_{n}F+G\right)
^{tk}\text{.}$$
Thus,   $\delta \in $ $\overline{\left( x_{n}F+G\right) ^{t}}$, and so  $%
\delta \in $ $\left( x_{n}F+G\right) ^{t}$.
Hence, one obtains $\delta \in (x_nF)^lG^h$ for some $l$ and $h$ with $l+h=t$, in particular, $\delta^{k} \in (x_nF)^{lk}G^{hk}$.
 It follows from the minimality of $p$ that $lk \geq p$. Because  $lk \geq p$ and $p+q =lk+hk=tk$, we gain
 $q \geq hk$.  Since $k\geq 1$, one can deduce from $q=bk$ and $q\geq hk$ that $b \geq h$.
It follows now from  $\delta \in (x_nF)^lG^h$  and  $b\geq h$  that    $\alpha=x_{n+1}^{b}\delta
 \in  \left(x_nF+x_{n+1}G\right)^{t}=L^t,$ and the argument is over.
\end{proof}

%%%%%%%%%%%%%%%%%%%%%%%%%%%
%%%%%%%%%%%%%%%%%%%%%%%%%%%

\section{On the closed neighborhood ideals and dominating ideals of complete bipartite graphs} \label{c-b graphs}

In this section, we provide some applications of the previous section. In fact,  we study the closed neighborhood ideals and dominating ideals of complete bipartite graphs.
We begin with the following theorem which says that
closed neighborhood ideals of complete bipartite graphs are normal, and so satisfy the (strong) persistence property.

 \begin{theorem} \label{NI-Bipartite}
 Let $K_{r,s}$ be a complete bipartite graph.
 Then the following statements hold:
\begin{itemize}
\item[(i)]   $NI(K_{r,s})$  is normal.
\item[(ii)]  $NI(K_{r,s})$  has the strong persistence property.
\item[(iii)]  $NI(K_{r,s})$  has the persistence property.
\end{itemize}
 \end{theorem}
 \begin{proof}
(i)  For convenience of notation, put $L:=NI(K_{r,s})$.
 Let $V(K_{r,s})=V_1 \cup V_2$, where $V_1=\{x_1, \ldots, x_r\}$ and   $V_2=\{x_{r+1}, \ldots, x_{r+s}\}$.
 If $r=1$, then $$L=(x_1\prod_{i=2}^{s+1}x_i, x_1x_2, \ldots, x_1x_{s+1})=x_1(x_2, \ldots, x_{s+1}).$$
 It is  well-known that any prime monomial ideal is normal, and by virtue of \cite[Remark 1.2]{ANR}, one has $L$ is normal.
 Similarly, if  $s=1$, then $L$ is normal too. Accordingly, we assume that $r,s>1$.
 Set
  $g:=\prod_{i=2}^rx_i$ and $f:=\prod_{i=r+1}^{r+s}x_i$. It is easy to check that
  $\mathcal{G}(L)=\{x_1f, \ldots, x_rf, x_1x_{r+1}g, \ldots, x_1x_{r+s}g\}.$ Hence, one can write $L=I+x_1H$, where
  $H:=(f, x_{r+1}g, \ldots, x_{r+s}g)$ and
  $I:=(x_2, \ldots, x_r)f$.
  A similar argument shows that $I$ is normal.   Since
 $H=g(x_{r+1}, \ldots, x_{r+s})+fR,$  where
  $R=K[x_1, \ldots, x_r, x_{r+1}, \ldots, x_{r+s}]$, it follows from \cite[Theorem 1.4]{ANR} that $H$ is normal.
  It is not hard to check that $I+H=H$, and so $I+H$ is normal as well. Note that $\gcd(x_1, v)=1$ for all $v\in \mathcal{G}(I) \cup \mathcal{G}(H)$.
   Now, the claim can be deduced immediately from Theorem \ref{I+hH}. \par
  (ii) and (iii) are trivial. This finishes the proof.
 \end{proof}

 %%%%%%%%%%%%%%%%%%%%%%%%%%%%%%%%%%%%%%%%%%%%%%%%%%

In the following proposition, we investigate when
the unique homogeneous maximal ideal appears in the associated primes set of powers of closed neighborhood ideals of complete bipartite graphs.

 \begin{proposition} \label{Maximal-Bipartite}
 Let $K_{r,s}$ with $r,s>1$ be a complete bipartite graph, $V(K_{r,s})=V_1 \cup V_2$, where $V_1=\{x_1, \ldots, x_r\}$ and   $V_2=\{x_{r+1}, \ldots, x_{r+s}\}$,  and
$$\mathfrak{m}=(x_1, \ldots, x_r, x_{r+1}, \ldots, x_{r+s})
\subset R=K[x_1, \ldots, x_r, x_{r+1}, \ldots, x_{r+s}],$$
be  the unique homogeneous maximal ideal.
 Then, for all $s\geq 3$,
 $$\mathfrak{m} \in \mathrm{Ass}(NI(K_{r,s})^s)\setminus
 \left(\mathrm{Ass}(NI(K_{r,s})) \cup \mathrm{Ass}(NI(K_{r,s})^2)\right).$$

 In particular,
 $\mathrm{depth}(R/NI(K_{r,s})^s)=0$ for all $s\geq 3$, and
 $$\mathrm{lim}_{k \rightarrow \infty}\mathrm{depth} (R/NI(K_{r,s})^k)=0.$$
 \end{proposition}

 \begin{proof}
 Set $L:=NI(K_{r,s})$,   $g:=\prod_{i=1}^rx_i$,  and $f:=\prod_{i=r+1}^{r+s}x_i$. It easy to see  that
  $\mathcal{G}(L)=\{x_1f, \ldots, x_rf, x_{r+1}g, \ldots, x_{r+s}g\}.$
  Due to   $L$ is a square-free monomial ideal, this implies that  $\mathrm{Ass}(L)=\mathrm{Min}(L)$, and hence  $\mathfrak{m}\notin \mathrm{Ass}(L).$
 We show that  $\mathfrak{m}\notin \mathrm{Ass}(L^2).$
  Suppose, on the contrary, that
 $\mathfrak{m}\in \mathrm{Ass}(L^2)$. This gives that
$\mathfrak{m}=(L^2:v)$ for some monomial $v\in R$.
     By virtue of \cite[Proposition 4.8]{SNQ}, one has
    $\mathrm{deg}_{x_i}(v) \leq 1$ for all $i=1, \ldots, r+s$.
   Let $x_d \in \mathfrak{m}$ for some $1\leq d \leq r+s$. Hence, $x_dv\in L^2$, and so there exists some monomial $w\in  \mathcal{G}(L^2)$ such that $w\mid x_dv$.
  In addition,  it follows from $\mathcal{G}(L)=\{x_1f, \ldots, x_rf, x_{r+1}g, \ldots, x_{r+s}g\}$ that
   there exist $i$ and $j$ with $1\leq i\neq j \leq r+s$ such that
   $\mathrm{deg}_{x_i}(w)=\mathrm{deg}_{x_j}(w)=2.$
 Since $\mathrm{deg}_{x_i}(v) \leq 1$ and $\mathrm{deg}_{x_j}(v) \leq 1$, this leads to a contradiction. Consequently, $\mathfrak{m}\notin \mathrm{Ass}(L^2)$.
  Based on  Theorem \ref{NI-Bipartite}(iii),  the ideal $L$ satisfies the persistence property. Hence, it is enough  for us to show that $\mathfrak{m}\in \mathrm{Ass}(L^3)$.
 To accomplish this,    set  $h:=\prod_{i=1}^{r+s}x^2_i$.
 We claim that $h\notin L^3$ and $\mathfrak{m} \subseteq (L^3:h)$.   We show that $x_k h \in L^3$ for all $1\leq k\leq r+s$. Let  $1\leq k\leq r$. Set
 $\alpha:=(x_{r+1}g)(x_{r+2}g)(x_kf)$.
 Because $\alpha \in L^3$ and $\alpha \mid x_kh$, one can derive that $x_kh\in L^3$. Now, let    $r+1\leq k\leq r+s$.
 Put $\beta:=(x_1f)(x_2f)(x_{k}g)$.
 Since $\beta \in L^3$ and $\beta \mid x_kh$, this implies that
 $x_kh\in L^3$. Therefore, we deduce that
 $\mathfrak{m} \subseteq (L^3:h)$. On the contrary assume that  $h\in L^3$. Thus, there exist monomials $u_1, u_2, u_3 \in \mathcal{G}(L)$ such that $u_1u_2u_3 \mid h$.  If
 $u_1, u_2, u_3 \in \{x_1f, \ldots, x_rf\}$ (respectively, $u_1, u_2, u_3 \in \{x_{r+1}g, \ldots, x_{r+s}g\}$), then $f^3 \mid h$ (respectively, $g^3 \mid h$), which contradicts the fact that $\mathrm{deg}_{x_i}(h)=2$ for all $i$.   Let
 $u_1, u_2 \in \{x_{r+1}g, \ldots, x_{r+s}g\}$,  and
 $u_3 \in  \{x_1f, \ldots, x_rf\}$, say $u_3=x_1f$. Then, we get $x_1^3 \mid u_1u_2u_3$, and so $x_1^3 \mid h$, which is a contradiction.
 Finally, let
 $u_1, u_2 \in \{x_1f, \ldots, x_rf\}$, and
 $u_3 \in   \{x_{r+1}g, \ldots, x_{r+s}g\}$, say $u_3=x_{r+1}g$. This gives rise to  $x_{r+1}^3 \mid u_1u_2u_3$, and so $x_{r+1}^3 \mid h$, which is a contradiction. Accordingly, $h\notin L^3$, and thus, $\mathfrak{m}=(L^3:h)$. This completes the proof.
  \end{proof}

 %%%%%%%%%%%%%%%%%%%%%%%%%%%%%%%%%%%%%%%%%%%%%%%%%%

We are in a position to state another main result of this section in the subsequent theorem. To see this, we have to use the corollary below as our main tool. 
For this purpose,  we first recall the definition of the monomial localization of a monomial ideal with respect to a monomial prime ideal. 
 Let $I$ be a monomial ideal in a polynomial ring $R=K[x_1, \ldots, x_n]$ over a field $K$.
   Let $\mathfrak{p}=(x_{i_1}, \ldots, x_{i_r})$ be a monomial prime ideal. The {\it monomial localization} of $I$ with respect to $\mathfrak{p}$, denoted by $I(\mathfrak{p})$, is the ideal in the polynomial ring $R(\mathfrak{p})=K[x_{i_1}, \ldots, x_{i_r}]$  which is obtained from $I$ by applying the $K$-algebra homomorphism $R\rightarrow R(\mathfrak{p})$  with $x_j\mapsto 1$  for all $x_j\notin \{x_{i_1}, \ldots, x_{i_r}\}$.
Moreover, it should be noted that $\mathfrak{m}\setminus \{x_i\}=(x_1, \ldots, x_{i-1}, x_{i+1}, \ldots, x_n)$, where $1\leq i \leq n$.

 \begin{corollary} \label{NNTF} \cite[Corollary 3.3]{NQKR}
Let $I$ be a square-free  monomial ideal in a polynomial ring $R=K[x_1, \ldots, x_n]$ over a field $K$.  Let   $I(\mathfrak{m}\setminus \{x_i\})$ be  normally torsion-free  for all $i=1, \ldots, n$. Then $I$ is nearly normally torsion-free.
\end{corollary}

\begin{theorem} \label{DI-Bipartite}
The dominating  ideals of  complete bipartite graphs are  nearly normally torsion-free.
\end{theorem}

\begin{proof}
 Let $K_{r,s}$ be a complete bipartite graph, $R=K[x_1, \ldots, x_r, x_{r+1}, \ldots, x_{r+s}]$, and
$\mathfrak{m}=(x_1, \ldots, x_r, x_{r+1}, \ldots, x_{r+s})$
be  the unique homogeneous maximal ideal.
 Let $V(K_{r,s})=V_1 \cup V_2$, where $V_1=\{x_1, \ldots, x_r\}$ and   $V_2=\{x_{r+1}, \ldots, x_{r+s}\}$. Put
 $L:=DI(K_{r,s})$, $I:=(x_i~:~i=1, \ldots, r)$, $J:=(x_i~:~i=r+1, \ldots, r+s)$,
  $g:=\prod_{i=1}^rx_i$, and $f:=\prod_{i=r+1}^{r+s}x_i$.
  According to \cite[Lemma 2.2]{SM}, and  \cite[Exercise 6.1.23]{V1},  we get the following equalities:
\begin{align*}
L=& \bigcap_{i=1}^r(J+x_iR) \cap \bigcap_{i=r+1}^{r+s}(I+x_iR) \\
=& (J+\bigcap_{i=1}^rx_iR) \cap (I+\bigcap_{i=r+1}^{r+s}x_iR) \\
=&J\cap I + \bigcap_{i=1}^rx_iR + \bigcap_{i=r+1}^{r+s}x_iR \\
=& JI+ gR + fR.
\end{align*}
In what follows, our aim is to use Corollary \ref{NNTF}. To do this, without loss of generality, it is enough to show that
$L(\mathfrak{m}\setminus \{x_1\})$ is  normally torsion-free.
Since $I(\mathfrak{m}\setminus \{x_1\})=R$,
  $J(\mathfrak{m}\setminus \{x_1\})=J$, $fR(\mathfrak{m}\setminus \{x_1\})=fR$, and
    $gR(\mathfrak{m}\setminus \{x_1\})=\prod_{i=2}^rx_iR$, we obtain
   $$L(\mathfrak{m}\setminus \{x_1\})=J+ \prod_{i=2}^rx_iR + fR=J+ \prod_{i=2}^rx_iR. $$
   It follows now from \cite[Theorem 2.5]{SN} that
   $L(\mathfrak{m}\setminus \{x_1\})$
    is normally torsion-free. Therefore, the claim can be deduced from Corollary \ref{NNTF}, and the proof is done.
\end{proof}

 %%%%%%%%%%%%%%%%%%%%%%%%%%%%%%%%%%%%%%%%%%%%%%%%%%
%%%%%%%%%%%%%%%%%%%%%%%%%%%%%%%%%%%%%%%%%%%%%%%%%%

 \section{Dominating ideals of $h$-wheel graphs}
\label{wheel}
The main aim of this section is to explore dominating ideals of  a class of graphs which are  called $h$-wheel graphs. For this purpose, we start by stating the definition of  $h$-wheel graphs.

%%%%%%%%%%%%%%%%%%%%%%%%%%%%%%%%%%%%%%%%%%%%%%%%%%

\begin{definition} \cite[Definition 1.6]{BHLSW}
A graph $G$ with the vertex set $V(G)$ is called an  {\it $h$-wheel}
if $V(G)$ can be written as the union of two disjoint sets, the set of rim vertices $R^G$ and the set of center vertices $C^G$, such that the following conditions hold:
\begin{itemize}
\item[(1)] The subgraph induced by $C^G$ is the complete  graph on $h$ vertices.
\item[(2)]  The subgraph induced by $R^G$ is an odd cycle.
\item[(3)] There exist $x_{i_1}, \ldots, x_{i_k} \in R^G$ with $k\geq 3$ such that $N_{R^G}(y)=\{x_{i_1}, \ldots, x_{i_k}\}$ for all $y\in C^G$.
\item[(4)] For every $y\in C^G$, the vertex $y$ belongs to at least two odd cycles in the subgraph induced by $y$ and $N_{R^G}(y)$.
\end{itemize}
\end{definition}

The $k$ is called the {\it radial number} for $G$. Also, for each $j=1, \ldots, k-1$, set $\ell_i$ as the length of the path along the subgraph induced by $R^G$ from $x_{i_j}$ to $x_{i_{j+1}}$, and set $\ell_k$ as the length from $x_{i_k}$ to $x_{i_1}$. The positive integers $\ell_1, \ldots, \ell_k$ are called the {\it radial lengths}.

It should be noted that, in \cite[Page  265]{KPS}, the authors studied the $1$-wheel, which we call a wheel for simplicity. In fact, given an $h$-wheel $G$ and a vetex $y\in C^G$, the subgraph induced by $y$ and $R^G$ is a  wheel.

%%%%%%%%%%%%%%%%%%%%%%%%%%%%%%%%%%%%%%%%%%%%%%%%
\begin{example}
We give a $4$-wheel graph $G$ in the following figure.

\begin{figure}[h!]
\centering
\begin{tikzpicture}[scale=.7]
\tikzset{vertex/.style = {shape = circle,fill = black,
inner sep = 0pt, outer sep = 0pt,minimum size = 7pt,}}
\node[label={[label distance=-0.01cm]180:$y_1$}] (y1) at (135:2) [vertex] {};
\node[label={[label distance=-0.01cm]0:$y_2$}] (y2) at (45:2) [vertex] {};
\node[label={[label distance=-0.01cm]0:$y_3$}] (y3) at (-45:2) [vertex] {};
\node[label={[label distance=-0.01cm]180:$y_4$}] (y4) at (-135:2) [vertex] {};
\draw[thick] (y1)--(y2)--(y3)--(y4)--(y1);
\draw[thick] (y1)--(y3);
\draw[thick] (y2)--(y4);

\node[label={[label distance=-0.01cm]90:$x_1$}] (x1) at (90:4) [vertex] {};
\node[label={[label distance=-0.01cm]180:$x_7$}] (x7) at (141:4) [vertex] {};
\node[label={[label distance=-0.01cm]180:$x_6$}] (x6) at (193:4) [vertex] {};
\node[label={[label distance=-0.01cm]270:$x_5$}] (x5) at (245:4) [vertex] {};
\node[label={[label distance=-0.01cm]270:$x_4$}] (x4) at (296:4) [vertex] {};
\node[label={[label distance=-0.01cm]0:$x_3$}] (x3) at (348:4) [vertex] {};
\node[label={[label distance=-0.01cm]0:$x_2$}] (x2) at (39:4) [vertex] {};
\draw[thick] (x1)--(x2)--(x3)--(x4)--(x5)--(x6)--(x7)--(x1);
\draw[thick] (y1)--(x1);
\draw[thick] (y1)--(x2);
\draw[thick] (y1)--(x5);
\draw[thick] (y1)--(x1);
\draw[thick] (y2)--(x1);
\draw[thick] (y2)--(x2);
\draw[thick] (y2)--(x5);
\draw[thick] (y3)--(x1);
\draw[thick] (y3)--(x2);
\draw[thick] (y3)--(x5);
\draw[thick] (y4)--(x1);
\draw[thick] (y4)--(x2);
\draw[thick] (y4)--(x5);
\node at(0,-6){4-$\text{wheel graph}$};
\end{tikzpicture}
\end{figure}

Note that $C^G=\{y_1, y_2, y_3, y_4\}$, $R^G=\{x_1, x_2, x_3, x_4, x_5, x_6, x_7\},$  and
$$N_{R^G}(y_1)=N_{R^G}(y_2)=N_{R^G}(y_3)=N_{R^G}(y_4)=\{x_1, x_2, x_5\}.$$

\end{example}

%%%%%%%%%%%%%%%%%%%%%%%%%
%%%%%%%%%%%%%%%%%%%%%%%%%

The following result is  necessary for us to  show Lemma \ref{Normality-H}.   We state it   here for ease of reference.

\begin{lemma} \label{Normal-Intersection} \cite[Lemma 2.5]{ANKRQ}
Suppose that $I$ and $J$ are two normal monomial
 deals in $R=K[x_1, \ldots, x_n]$ such that   $\mathrm{gcd}(u,v)=1$ for all $u\in \mathcal{G}(I)$ and $v\in \mathcal{G}(J)$.  Then $I\cap J=IJ$ is normal.
\end{lemma}

%%%%%%%%%%%%%%%%%%%%%%%%%%%%
%%%%%%%%%%%%%%%%%%%%%%%%%%%%

%%%%%%%%%%%%%%%%%%%%%%%%%%%%
%%%%%%%%%%%%%%%%%%%%%%%%%%%%

\begin{lemma}\label{Normality-H}
Let $C_n$ be a cycle graph  with the vertex set $V(C_n)=\{1, \ldots, n\}$ and the edge set  $E(C_n)=\{\{x_i, x_{i+1}\}~:~ i=1, \ldots, n\},$
where $x_0$ (respectively, $x_{n+1}$) represents $x_n$ (respectively, $x_1$).
 Let $$H:=\bigcap_{j\in \{1, \ldots, n\}\setminus \{\ell_1, \ldots, \ell_k\}} (x_{j-1}, x_{j}, x_{j+1}).$$
 Then the following statements hold:
\begin{itemize}
\item[(i)]   $H$  is normal.
\item[(ii)]  $H$  has the strong persistence property.
\item[(iii)]  $H$  has the persistence property.
\end{itemize}
\end{lemma}

\begin{proof}
 (i) Clearly $n\geq 3$. Let $\{1, \ldots, n\}\setminus \{\ell_1, \ldots, \ell_k\} =\{i_1, \ldots, i_{n-k}\}$.
  To simplify the notation, set $u_j:=(x_{i_j-1}, x_{i_j}, x_{i_j+1})\subset R$ for each $j=1, \ldots, n-k$, where $R=K[x_1, \ldots, x_n]$. We use induction on $s:=n-k$.
 If $s=0$, then there is nothing to prove. Let $s=1$. Then $H=u_1=(x_{i_1-1}, x_{i_1}, x_{i_1+1})$, which is certainly normal. Now, let $s>1$, and the claim has been shown for all integers less than  $s$. Put $J:=\cap_{j=1}^{s-1}u_j$. Note that  the inductive hypothesis implies  $J$ is normal. Moreover,  $H=J\cap u_s$.
 Here, one may consider the following cases:
  \vspace{1mm}

\textbf{Case 1.}  $\mathrm{supp}(J) \cap \mathrm{supp}(u_s)=\emptyset$. Then   Lemma \ref{Normal-Intersection} yields that $H$ is normal.
\vspace{1mm}

\textbf{Case 2.}  $|\mathrm{supp}(J) \cap \mathrm{supp}(u_s)|=1$. According to  Corollary  \ref{Cor.vI+wJ}, one can derive that $H$ is normal.
\vspace{1mm}

\textbf{Case 3.} $|\mathrm{supp}(J) \cap \mathrm{supp}(u_s)|=2$, say $\mathrm{supp}(J) \cap \mathrm{supp}(u_s)=\{x_{a}, x_{b}\}$.    In this case, one may  consider the following subcases:
\vspace{1mm}

\textbf{Subcase 3.1.}  $|a-b|=1$, say $\mathrm{supp}(J) \cap \mathrm{supp}(u_s)=\{x_{i_s-1}, x_{i_s}\}$. This implies that
$H=J\cap (x_{i_s-1}, x_{i_s}) +  x_{i_s+1}J$.
It is routine to check that one can write
$$\displaystyle J\cap (x_{i_s-1}, x_{i_s})=\bigcap_{j=1,~ \text{where}~ x_{i_s-1} x_{i_s}\notin u_j}^{s-1}u_j \cap (x_{i_s-1},x_{i_s}).$$
The inductive hypothesis yields that
$\cap_{j=1,~ \text{where}~ x_{i_s-1} x_{i_s}\notin u_j}^{s-1}u_j$ is normal.
Notice that any two consecutive integers can be appeared at most in two $u_i$'s. Also, recall that any integer can be appeared at most in three $u_i$'s.
It follows now from   Corollary  \ref{Cor.vI+wJ} that
$\cap_{j=1,~ \text{where}~ x_{i_s-1} x_{i_s}\notin u_j}^{s-1}u_j\cap (x_{i_s-1}, x_{i_s})$ is normal too.
On account of  Theorem   \ref{I+hH}(i), we deduce  that $H$ is normal.
\vspace{1mm}

\textbf{Subcase 3.2.}  $|a-b|=2$, that is,  $\mathrm{supp}(J) \cap \mathrm{supp}(u_s)=\{x_{i_s-1}, x_{i_s+1}\}$. We thus have
$H=J\cap (x_{i_s-1}, x_{i_s+1}) +  x_{i_s}J$.
Want to show that  $J\cap (x_{i_s-1}, x_{i_s+1})$    is normal.
Our strategy is to use  Proposition  \ref{xn&x(n+1)}.    For this purpose, it is sufficient  to prove that  $J\cap (x_{i_s-1}) + (J:x_{i_s+1})$ is normal.
If $({x_{i_s-3}, x_{i_s-2}, x_{i_s-1}})$ or $({x_{i_s+1}, x_{i_s+2}, x_{i_s+3}})$ or both of them do not appear in $J$, then Corollary  \ref{Cor.vI+wJ}   yields that $J\cap (x_{i_s-1}, x_{i_s+1})$    is normal. Hence, assume that $({x_{i_s-3}, x_{i_s-2}, x_{i_s-1}})$ and  $({x_{i_s+1}, x_{i_s+2}, x_{i_s+3}})$   appear in $J$.  This means that we can write
$$J=J_1 \cap ({x_{i_s-3}, x_{i_s-2}, x_{i_s-1}})\cap  ({x_{i_s+1}, x_{i_s+2}, x_{i_s+3}}).$$
Note that the inductive hypothesis gives  that  $J_1$ is normal.
Now, we obtain
\begin{align*}
J \cap (x_{i_s-1}) + (J:x_{i_s+1})= &
J_1 \cap    (x_{i_s+1}, x_{i_s+2}, x_{i_s+3}) \cap (x_{i_s-1}) \\
& + J_1 \cap (x_{i_s-3}, x_{i_s-2}, x_{i_s-1})\\
= &J_1 \cap (x_{i_s-3}, x_{i_s-2}, x_{i_s-1}),
\end{align*}
 where $\mathrm{supp}(J_1) \cap \{x_{i_s-1}, x_{i_s+1}\}=\emptyset$.
Since $x_{i_s-1} \notin J_1$, this implies  that  the normality of $J_1 \cap (x_{i_s-3}, x_{i_s-2}, x_{i_s-1})$ can be concluded  from Case 1 or  Case 2 or Subcase 3.1. Therefore, $J \cap (x_{i_s-1}) + (J:x_{i_s+1})$ is normal. It follows from  Proposition  \ref{xn&x(n+1)} that $J\cap (x_{i_s-1}, x_{i_s+1})$    is normal. On account of  Theorem   \ref{I+hH}(i), we deduce  that $H$ is normal.

\vspace{1mm}
\textbf{Case 4.} $\mathrm{supp}(J) \cap \mathrm{supp}(u_s)=\{x_{i_s-1}, x_{i_s}, x_{i_s+1}\}$.
In this case, we can consider the following subcases:
\vspace{1mm}

\textbf{Subcase  4.1.} There exists exactly one integer $\alpha \in \{1, \ldots, s-1\}$ such that
$|\mathrm{supp}(u_\alpha) \cap \mathrm{supp}(u_s)|=2$, say $\mathrm{supp}(u_\alpha) \cap \mathrm{supp}(u_s)=\{x_{i_{s}-1}, x_{i_s}\}$. Also, there exists unique
$\beta \in \{1, \ldots, s-1\}\setminus \{\alpha\}$ such that
$x_{i_s+1} \in u_\beta$.
Assume that   $u_\beta = (x_{i_s+1}, x_r, x_t)$. Set  $A:=\cap_{j\in \{1, \ldots, s-1\}\setminus \{\beta\}}u_j$. We thus get
\begin{align*}
H=& A \cap \left((x_{i_s+1}) + (x_{i_s-1}, x_{i_s}) \cap (x_r, x_t)\right) \\
=& A\cap (x_{i_s-1}, x_{i_s}) \cap (x_r, x_t) + x_{i_s+1}A.
\end{align*}
It follows from the inductive hypothesis that $A$ is normal. Furthermore, by mimicking the argument in Case 3, one can derive that  $A\cap (x_{i_s-1}, x_{i_s}) \cap (x_r, x_t)$ is normal. We therefore conclude from Theorem \ref{I+hH}(i) that $H$ is normal.
\vspace{1mm}

\textbf{Subcase 4.2.}  There exist  exactly two distinct  integers
 $\alpha, \beta \in \{1, \ldots, s-1\}$ such that
$|\mathrm{supp}(u_\alpha) \cap \mathrm{supp}(u_s) \cap \mathrm{supp}(u_\beta)|=1$, say
$\mathrm{supp}(u_\alpha) \cap \mathrm{supp}(u_s) \cap \mathrm{supp}(u_\beta)=x_{i_s}.$
Let  $u_\alpha = (x_{i_s}, x_{\theta_1}, x_{\theta_2})$ and
$u_\beta = (x_{i_s}, x_{\lambda_1}, x_{\lambda_2})$.
Put   $B:=\cap_{j\in \{1, \ldots, s-1\}\setminus \{\alpha, \beta\}}u_j$.  This yields that
\begin{align*}
H=& B \cap \left((x_{i_s}) +   (x_{i_s-1}, x_{i_s+1}) \cap (x_{\theta_1}, x_{\theta_2}) \cap (x_{\lambda_1}, x_{\lambda_2})\right) \\
=& B\cap (x_{i_s-1}, x_{i_s+1}) \cap (x_{\theta_1}, x_{\theta_2}) \cap (x_{\lambda_1}, x_{\lambda_2}) + x_{i_s}B.
\end{align*}
One can deduce from the  inductive hypothesis that $B$ is normal. By repeating the argument in Case 3, we can deduce that
$B\cap (x_{i_s-1}, x_{i_s+1}) \cap (x_{\theta_1}, x_{\theta_2}) \cap (x_{\lambda_1}, x_{\lambda_2})$ is normal as well.
Thanks to Theorem \ref{I+hH}(i), we obtain  $H$ is normal.

This completes the inductive step, and hence  the claim  has been proved by induction.   \par
(ii) and (iii) are trivial.
\end{proof}

%%%%%%%%%%%%%%%%%%%%%%%%%%
%%%%%%%%%%%%%%%%%%%%%%%%%%
It has already been established in \cite[Theorem 3.3]{NKA} and
\cite[Theorem 1.10]{ANR} that the cover ideals of odd cycle graphs are normal and have the (strong) persistence property. On the other hand, since any even cycle is a bipartite graph, its cover ideal is normally torsion-free, and so is normal and also has the (strong) persistence property.
The next corollary which  is an immediate  consequence of Lemma \ref{Normality-H}, says that the dominating ideals of cycle graphs are normal and also satisfy the (strong) persistence property. Moreover, this corollary  will be used in proving   Theorem \ref{DI-Wheel}.

\begin{corollary} \label{Normality-Cycle}
Let $C_n$ be a cycle graph  with the vertex set $V(C_n)=\{1, \ldots, n\}$ and the edge set  $E(C_n)=\{\{x_i, x_{i+1}\}~:~ i=1, \ldots, n\},$
where  $x_{n+1}$  represents  $x_1$, and $DI(C_n)$ be its dominating ideal.
 Then the following statements hold:
\begin{itemize}
\item[(i)]   $DI(C_n)$  is normal.
\item[(ii)]  $DI(C_n)$  has the strong persistence property.
\item[(iii)]  $DI(C_n)$  has the persistence property.
\end{itemize}
 \end{corollary}

%%%%%%%%%%%%%%%%%%%%%%%%%%%%
%%%%%%%%%%%%%%%%%%%%%%%%%%%%
It has already been introduced, in \cite{BBV}, the notion of
partial $t$-cover ideals of  finite simple graphs. We first recall the definition of  partial $t$-cover ideals  in the following definition.
\begin{definition} \cite[Definition 1.1]{BBV}
Suppose that $G$ is a finite simple
graph on the vertex set $V(G) = \{x_1, x_2, \ldots, x_n\}$
 with the edge set $E(G)$. Also, for any $x\in V(G)$,  let
$N(x) = \{y~:~ \{x, y\}\in  E(G)\}$ denote the set of neighbors of $x$. Fix an integer $t\geq 1$. The {\it partial $t$-cover ideal} of $G$  is the monomial ideal
$$J_t(G)=\bigcap_{x\in V(G)}\left(\bigcap_{\{x_{i_1}, \ldots, x_{i_t}\} \subseteq N(x)}(x, x_{i_1}, \ldots, x_{i_t})\right).$$
\end{definition}

%%%%%%%%%%%%%%%%%%%%%%%%%%%%
%%%%%%%%%%%%%%%%%%%%%%%%%%%%
When $t=1$, above construction is simply the cover ideal of a finite simple graph $G$. It has already been shown in \cite[Theorem 1.2]{BBV} that if $T$ is a tree, then, for any $t\geq 1$, $J_t(T)$ satisfies the persistence property. The next corollary states that the partial $2$-cover ideal of any cycle graph satisfies the persistence property as well.
\begin{corollary}
Let $C_n$ be a cycle graph  and $J_2(C_n)$ be its partial $2$-cover ideal.
 Then the following statements hold:
\begin{itemize}
\item[(i)]   $J_2(C_n)$  is normal.
\item[(ii)]  $J_2(C_n)$  has the strong persistence property.
\item[(iii)]  $J_2(C_n)$  has the persistence property.
\end{itemize}
 \end{corollary}

%%%%%%%%%%%%%%%%%%%%%%%%%%%%%%%%%%%%%%%%%%%%%%%%%%%%%%%%
%%%%%%%%%%%%%%%%%%%%%%%%%%%%%%%%%%%%%%%%%%%%%%%%%%%%%%%%

We are ready to express the main result of this section in the following theorem.

\begin{theorem} \label{DI-Wheel}
Let $G$ be an $h$-wheel graph  with  rim  $R^G$ and   center  $C^G$ such that
 $V(C^G)=\{y_1, \ldots, y_h\}$ and $V(R^G)=\{x_1, \ldots, x_{2m+1}\}$, where $m\geq 2$.  Also, let  $x_{\ell_1}, \ldots, x_{\ell_k}$ be the radial vertices such that  there exist at least three consecutive numbers among them. Let $DI(G)$ denote the  dominating ideal of $G$. Then the following statements hold:
\begin{itemize}
\item[(i)]   $DI(G)$  is normal.
\item[(ii)]  $DI(G)$  has the strong persistence property.
\item[(iii)]  $DI(G)$  has the persistence property.
\end{itemize}
\end{theorem}

 \begin{proof}
 (i)   To simplify the notation, put $F:=(x_{\ell_1}, \ldots, x_{\ell_k})S$ and
 $J:=(y_1, \ldots, y_h)S$, where $S=K[x_1, \ldots, x_{2m+1}, y_1, \ldots, y_h]$. Since there exist at least three consecutive numbers among $\ell_1, \ldots, \ell_k$, this gives rise to
 $\cap_{i=1}^k (x_{\ell_i-1}, x_{\ell_i}, x_{\ell_i+1}) \subseteq F$.
 Hence, by remembering this fact that    $x_0$ (respectively, $x_{2m+2}$) represents $x_{2m+1}$ (respectively, $x_1$),   one can deduce  the following equalities
 \begin{align*}
 DI(G)=& (F+J) \cap \bigcap_{i=1}^k \left( (x_{\ell_i-1}, x_{\ell_i}, x_{\ell_i+1})+J\right) \\
 & \cap  \bigcap_{j\in \{1, \ldots, 2m+1\}\setminus \{\ell_1, \ldots, \ell_k\}} (x_{j-1}, x_{j}, x_{j+1})  \\
 =&  \left(\bigcap_{i=1}^k (x_{\ell_i-1}, x_{\ell_i}, x_{\ell_i+1})+J\right)\cap
 \bigcap_{j\in \{1, \ldots, 2m+1\}\setminus \{\ell_1, \ldots, \ell_k\}} (x_{j-1}, x_{j}, x_{j+1})  \\
 =&  \bigcap_{j=1}^{2m+1} (x_{j-1}, x_{j}, x_{j+1}) +J\bigcap_{j\in \{1, \ldots, 2m+1\}\setminus \{\ell_1, \ldots, \ell_k\}} (x_{j-1}, x_{j}, x_{j+1})\\
  =&  DI(R^G) + J\bigcap_{j\in \{1, \ldots, 2m+1\}\setminus \{\ell_1, \ldots, \ell_k\}} (x_{j-1}, x_{j}, x_{j+1}).
  \end{align*}
 Set $I:= DI(R^G)$ and $H:=\bigcap_{j\in \{1, \ldots, 2m+1\}\setminus \{\ell_1, \ldots, \ell_k\}} (x_{j-1}, x_{j}, x_{j+1})$. Hence, one has $DI(G)=I+JH$.
 By virtue of  Lemma   \ref{Normality-H},   we obtain $H$ is normal.  Moreover, Corollary \ref{Normality-Cycle} yields that $I$ is normal.  Because $I \subseteq H$, $\mathrm{gcd}(u,v)=1$  for all $v, u \in \mathcal{G}(J)$ with $u\neq v$, and also  $\mathrm{gcd}(u,v)=1$ for all $u\in \mathcal{G%
}(I)\cup \mathcal{G}(H)$ and $v\in \mathcal{G}(J)$, one can conclude immediately from Proposition \ref{I+JH}(i) that $DI(G)$ is normal, as desired. \par
(ii) and (iii) are trivial.
 \end{proof}

%%%%%%%%%%%%%%%%%%%%%%%%%
%%%%%%%%%%%%%%%%%%%%%%%%%

 \noindent{\bf Acknowledgments.}

 We would like to thank the anonymous referee for her/his careful reading of the manuscript and useful comments.
Part of this work was carried out while Mehrdad Nasernejad, in 2021, 
was invited CNRS Researcher at the F\'e{d}\'e{r}ation de Recherche Math\'e{m}atique 
des Hauts-de-France (FR2037 CNRS), working at the Laboratoire de 
Mathématiques de Lens of the Universit\'e  d'Artois (France).

%%%%%%%%%%%%%%%%%%%%%%%%%%%%
%%%%%%%%%%%%%%%%%%%%%%%%%%%%


\begin{thebibliography}{999}

\bibitem{ANKRQ} I. Al-Ayyoub,  M. Nasernejad,
K.  Khashyarmanesh,  L. G. Roberts, and V. C.  Qui$\mathrm{\tilde{n}}$onez,  { \it Results on the normality of  square-free monomial ideals and cover ideals under some graph operations}, Math. Scand. {\bf 127} (2021), 441--457.



\bibitem{ANR} I. Al-Ayyoub, M. Nasernejad, and L. G.  Roberts, {\it Normality of cover ideals of graphs and
normality under some operations}, Results Math.  {\bf 74}  (4) (2019) 26 pages.

 \bibitem{ANR1}  I. Al-Ayyoub,  M. Nasernejad, and L. G. Roberts, {\it  On the normality and associated primes  of cover ideals of a class of imperfect graphs,}
Mathematical Reports, 2023, To appear. 

  \bibitem{ANR2}  I. Al-Ayyoub,  M. Nasernejad, and L. G. Roberts, {\it  On the strong persistence property and normality of cover
ideals of theta graphs,} 2023, Comm. Algebra, 2023, \doi{https://doi.org/10.1080/00927872.2023.2188413}.  

\bibitem{BBV} A. Bhat, J. Biermann, and A. Van Tuyl, {\it Generalized cover ideals and the persistence property,}
 J. Pure Appl. Algebra {\bf 218} (2014), no. 9, 1683--1695.



\bibitem{BHLSW}  C. Brooke, M. Hoch, S. Lato,  J.  Striuli, and  B. Wang, {\it Associated primes of $h$-wheels,}
 Involve {\bf 12} (2019), no. 3, 411--425.


\bibitem{BM}  J. A. Bondy and U. S. R. Murty, \textit{Graph theory.} Graduate Texts in Mathematics, \textbf{244}. Springer, New York, 2008. xii+651 pp.

\bibitem{BR}  M. Brodmann, \textit{Asymptotic stability of $\mathrm{Ass}(M/I^{n}M$).} Proc. Amer. Math. Soc.
\textbf{74} (1979), 16--18.


\bibitem{Claudia}  C. Andrei-Ciobanu, {\it Nearly normally torsionfree ideals,}
Combinatorial Structures in Algebra and Geometry, 1--13, NSA 26, Constanta, Romania, August 26-September 1, 2018.



\bibitem{FHV2}  C. Francisco, H. T. H$\mathrm{\grave{a}}$, and A. Van Tuyl,  \textit{Colorings of hypergraphs, perfect graphs and associated primes of powers of monomial ideals.}  J. Algebra   \textbf{331} (2011), 224--242.


\bibitem{HH1} J. Herzog and T. Hibi, {Monomial Ideals,} Graduate Texts in Mathematics \textbf{260} Springer-Verlag, 2011.


\bibitem{HQ}  J. Herzog  and A. A. Qureshi, {\it  Persistence and stability properties of  powers of ideals,}   J. Pure Appl. Algebra { \bf 219} (2015),  530--542.

\bibitem{HS} J. Honeycutt  and S. K. Sather-Wagstaff, {\it Closed neighborhood ideals of finite simple  graphs},  Matematica {\bf 1} (2022), No. 2, 387--394. 


\bibitem{KPS} K. Kesting, J. Pozzi, and J. Striuli, {\it  On the associated primes of the third order of the cover ideal, }
 Involve {\bf 4} (2011), no. 3, 263--270.



\bibitem{KNT} K. Khashyarmanesh, M. Nasernejad,   and  J. Toledo, {\it Symbolic strong persistence property  under monomial operations and strong persistence property of cover ideals.}  Bull. Math. Soc. Sci. Math. Roumanie (N.S.) \textbf{ 64 (112)}, No. 2,  (2021),  103--129.

\bibitem{N2}  M. Nasernejad, {\it  Persistence property for some  classes of  monomial ideals of a polynomial ring,} J. Algebra Appl. \textbf{16}, No. 5 (2017) 1750105 (17 pages).



\bibitem{NKA}  M. Nasernejad,  K. Khashyarmanesh, and  I. Al-Ayyoub,  \textit{Associated primes of powers  of cover ideals under graph operations.}
 Comm. Algebra  \textbf{47} (2019), no. 5, 1985--1996.

\bibitem{NKRT}   M. Nasernejad,   K. Khashyarmanesh, L. G. Roberts, and J. Toledo, \textit{The strong persistence property and symbolic strong persistence property.} Czechoslovak Math. J. 72(147) (2022), no. 1, 209--237. 

\bibitem{NQ} M. Nasernejad and  A. A. Qureshi, {\it Algebraic implications of neighborhood hypergraphs and their transversal hypergraphs},
  Submitted. \url{ArXiv:2208.14040v1}.  


\bibitem{NQBM} M. Nasernejad, A. A. Qureshi, S. Bandari, and A.  Musapa\c{s}ao\u{g}lu, {\it  Dominating ideals and closed neighborhood ideals of graphs},  
 Mediterr. J. Math. 19 (2022), no. 4, Paper No. 152, 18 pp.



\bibitem{NQKR}  M. Nasernejad, A. A. Qureshi, K. Khashyarmanesh, and L. G. Roberts,  {\it Classes of normally and nearly normally torsion-free monomial ideals,} Comm. Algebra 50 (2022), no. 9, 3715--3733.

\bibitem{RNA}  S. Rajaee, M. Nasernejad, and  I. Al-Ayyoub, {\it Superficial ideals for monomial ideals,} J. Algebra Appl. {\bf 16}(2), (2018) 1850102 (28 pages).

\bibitem{RT}  E. Reyes  and J. Toledo,  \textit{On the strong persistence property for monomial ideals.}
Bull. Math. Soc. Sci. Math. Roumanie (N.S.)  \textbf{60(108)} (2017),  293--305.



\bibitem{SN}  M. Sayedsadeghi and M. Nasernejad,
{\it Normally torsion-freeness of monomial ideals under monomial operators,}  Comm. Algebra {\bf 46}(12), 5447--5459 (2018).

\bibitem{SNQ}  M. Sayedsadeghi,  M. Nasernejad, and A. A. Qureshi,  {\it On the embedded associated primes of monomial ideals,} 
Rocky Mountain J. Math. 52 (2022), no. 1, 275--287. 

\bibitem{SM} L. Sharifan and S. Moradi,  {\it Closed neighborhood ideal of a graph},
Rocky Mountain J. Math. {\bf 50}(3), (2020) 1097--1107.


 \bibitem{SVV} A. Simis, W. Vasconcelos, and R. Villarreal,
\textit{On the ideal theory of graphs.}  J. Algebra \textbf{167} (1994), 389--416.

 \bibitem{V}  R.H. Villarreal, {\it Cohen--Macaulay graphs,}  Manuscripta Math. {\bf 66} (1990), 277--293.

 \bibitem{V1} R. H. Villarreal, {Monomial Algebras.} 2nd. Edition, Monographs and Research Notes in Mathematics, CRC Press, Boca Raton, FL, 2015.


\end{thebibliography}
\end{document}